\documentclass[runningheads]{llncs}
\usepackage[T1]{fontenc}
\usepackage{graphicx}
\usepackage{amsmath}
\usepackage{amssymb}
\usepackage{mathtools}
\usepackage{bbm}
\usepackage{hyperref}
\usepackage{subcaption}
\usepackage[textsize=tiny]{todonotes}

\newcommand{\R}{\mathbb{R}}

\usepackage{color}

\begin{document}

\title{Learning Hamiltonian Systems with Mono-Implicit Runge–Kutta Methods\thanks{Supported by the Research Council of Norway, through the project DynNoise: Learning dynamical systems from noisy data. (No.
339389).}}

%
%
\author{Håkon Noren}
\authorrunning{H. Noren}
%
\institute{Department of Mathematical Sciences, Norwegian University of Science and Technology, Trondheim, Norway \\
\email{hakon.noren@ntnu.no}
}
\maketitle              

    \begin{abstract}
    Numerical integrators could be used to form interpolation conditions when training neural networks to approximate the vector field of an ordinary differential equation (ODE) from data. When numerical one-step schemes such as the Runge--Kutta methods are used to approximate the temporal discretization of an ODE with a known vector field, properties such as symmetry and stability are much studied. Here, we show that using mono-implicit Runge--Kutta methods of high order allows for accurate training of Hamiltonian neural networks on small datasets. This is demonstrated by numerical experiments where the Hamiltonian of the chaotic double pendulum in addition to the Fermi--Pasta--Ulam--Tsingou system is learned from data.

\end{abstract}

\section{Introduction}

In this paper, we apply backward error analysis \cite{hairer2006geometric} to motivate the use of numerical integrators of high order when approximating the vector field of ODEs with neural networks. We particularly consider mono-implicit Runge--Kutta (MIRK) methods \cite{cash1975class,burrage1994order}, a class of one-step methods that are explicit when solving inverse problems. Such methods can be constructed to have high order with relatively few stages, compared to explicit Runge--Kutta methods, and attractive properties such as symmetry. Here, we perform numerical experiments learning two Hamiltonian systems with MIRK methods up to order $p=6$. To the best of our knowledge, this is the first demonstration of the remarkable capacity of numerical integrators of order $p>4$ to facilitate the training of Hamiltonian neural networks \cite{HamiltonianNeuralNetworks} from sparse datasets, to do accurate interpolation and extrapolation in time.


Recently, there has been a growing interest in studying neural networks through the lens of dynamical systems. This is of interest both to accelerate data-driven modeling and for designing effective architectures for neural networks \cite{haber2017stable,ruthotto2020deep,machine_dynamical}. Considering neural network layers as the flow of a dynamical system is the idea driving the study of so-called neural ODEs \cite{chen2018neural} and its discretized counter-part, residual neural networks. 

Hamiltonian mechanics provide an elegant formalism that allows a wide range of energy preserving dynamical systems to be described as first order ODEs. Hamiltonian neural networks \cite{HamiltonianNeuralNetworks} aim at learning energy-preserving dynamical systems from data by approximating the Hamiltonian using neural networks.  A central issue when studying neural networks and dynamical systems is which method to use when discretizing the continuous time dynamics. Several works use backward error analysis to argue for the importance of using symplectic integrators for learning the vector field of Hamiltonian systems \cite{Chen2020Symplectic,DeepHamiltoniannetworksbasedonsymplecticintegrators,offen2022symplectic}. Using Taylor expansions to derive the exact form of the inverse modified vector field allows for the construction of a correction term that cancels the error stemming from the temporal discretization, up to arbitrary order \cite{offen2022symplectic,SymplecticLearningforHamiltonianNeuralNetworks}.

\section{Inverse ODE problems on Hamiltonian form}
  We consider a first-order ODE 
  \begin{equation}\label{eq:ode}
    \frac{d}{dt} y(t) = f(y(t)), \quad y(t) : [0,T] \rightarrow \R^n,
    \end{equation}
  and assume that the vector field $f$ is unknown, whereas samples $S_N = \{y(t_n)\}_{n=0}^N$ of the solution are available, with constant step size $h$. Then the inverse problem aims at deriving an approximation $f_{\theta} \approx f$ where $\theta$ is a set of parameters to be chosen. The inverse problem can be formulated as the following optimization problem:
\begin{equation}
  \begin{aligned}
  \operatorname*{arg\,min}_\theta  \sum_{n=0}^{N-1} \bigg \| y(t_{n+1}) - \Phi_{h,f_{\theta}}(y(t_n))   \bigg \|,
  \label{inverse_problem}
  \end{aligned}
\end{equation}
where $f_{\theta}$ is a neural network approximation of $f$ with parameters $\theta$, and $\Phi_{h,f_{\theta}}$ is a one-step integration method with step size $h$ such that  $y_{n+1} = \Phi_{h,f}(y_n)$.  In particular, we assume that \eqref{eq:ode} is a Hamiltonian system, meaning that 
\begin{equation}
  \label{hamiltonian system}
    f(y) = J\nabla H(y(t)), \quad J := \begin{bmatrix}
      0 & I\\
      -I & 0
    \end{bmatrix} \in \R^{2d \times 2d}.
  \end{equation}
We follow the idea of Hamiltonian neural networks \cite{HamiltonianNeuralNetworks} aiming at approximating the Hamiltonian, $H : \R^{2d} \rightarrow \R$, such that $H_{\theta}$ is a neural network and $f$ is approximated by  $f_{\theta}(y) := J\nabla H_{\theta}(y)$. It thus follows that the learned vector field 
$f_{\theta}$ by construction is Hamiltonian.

\section{Mono-implicit Runge--Kutta for inverse problems}
Since the solution is known point-wise, $S_N = \{ y(t_n)\}_{n=0}^N$, the points $y_n$ and $y_{n+1}$ can be substituted by $y(t_{n})$ and $y(t_{n+1})$ when computing the next step of a one-step integration method. We denote this substitution as the \textit{inverse injection}, and note that is yields an interpolation condition for $f_{\theta}\approx f$ for each $n$. If we let $\Phi_{h,f_{\theta}}$ in \eqref{inverse_problem} be the so-called implicit midpoint method, we get the following expression to be minimized:
\begin{equation}
  \begin{aligned}
 \bigg \| y(t_{n+1}) - \bigg ( y(t_n) + hf_{\theta}\big( \frac{y(t_n) +  y(t_{n+1})}{2} \big) \bigg )   \bigg \|, \quad n = 0\dots,N-1. 
  \label{midpoint_opt}
  \end{aligned}
\end{equation}
For the midpoint method, the inverse injection bypasses the computationally costly problem of solving a system of equations within each training iteration, since $y(t_{n+1})$ is known. More generally, mono-implicit Runge--Kutta (MIRK) methods constitute the class of all Runge--Kutta methods that form explicit methods under this substitution. Given vectors $b,v\in\R^s$ and a strictly lower triangular matrix $D \in \R^{s\times s}$, a MIRK method is a Runge--Kutta method where $A = D + vb^T$, and is thus given by
\begin{equation}
\begin{aligned}
    \label{mirk-stages}
        y_{n+1} &= y_n + h\sum_{i=1}^s b_i k_i,   \\
    k_i &= f\big (y_n + v_i(y_{n+1} - y_n) + h\sum_{j=1}^s d_{ij} k_j \big).
  \end{aligned} 
\end{equation} 
Let us denote $\hat y_{n+1}$ and $\hat k_i$ as the next time-step and the corresponding stages of a MIRK method when substituting $y_n,y_{n+1}$ by $y(t_n),y(t_{n+1})$ on the right-hand side of \eqref{mirk-stages}.
\begin{theorem}
    Let $y_{n+1}$ be given by a MIRK scheme \eqref{mirk-stages} of order $p$ and $\hat y_{n+1}$ be given by the same method under the inverse injection. Assume that only one integration step is taken from a known initial value $y_n = y(t_n)$. Then 
    \begin{align}
        \hat y_{n+1}  &= y_{n+1} +  \mathcal O(h^{p+2})\\
       \text{and}\quad  \hat y_{n+1}  &=  y(t_{n+1}) +  \mathcal O(h^{p+1}).
    \end{align}
\end{theorem}

\begin{proof}
    Since the method \eqref{mirk-stages} is of order $p$ we have that 
    \begin{align*}
        \hat k_1 &= f\big (y(t_{n}) + v_1(y(t_{n+1}) - y(t_{n})) \big)\\
        &= f\big (y_{n} + v_1(y_{n+1} - y_{n})\big) + \mathcal O(h^{p+1}) 
        = k_1 + \mathcal O(h^{p+1})
      \end{align*} 
The same approximation could be made for $\hat k_2,\dots,\hat k_s$, since $D$ is strictly lower triangular, yielding $\hat k_i = k_i + \mathcal O(h^{p+1})$ for $i = 1,\dots,s$. In total, we find that
\begin{align*}
    \hat y_{n+1} &= y(t_n) + h\sum_{i=1}^s b_i \hat k_i   \\
    &= y_n + h\sum_{i=1}^s b_i k_i + \mathcal O(h^{p+2}) \\
    &= y_{n+1} + \mathcal O(h^{p+2})\\
    &= y(t_{n+1}) + \mathcal O(h^{p+1}) + \mathcal O(h^{p+2}).\\
    &= y(t_{n+1}) + \mathcal O(h^{p+1}).
\end{align*}
\end{proof}

For the numerical experiments, we will consider the optimal  MIRK methods derived in \cite{muir_optimal_nodate}. The minimal number of stages required to obtain order $p$ is $s = p-1$ for MIRK methods  \cite{burrage1994order}. In contrast, explicit Runge--Kutta methods need $s=p$ stages to obtain order $p$ for $1\leq p\leq 4$ and $s = p+1$ stages for $p = 5,6$ \cite{butcher2016numerical}, meaning that the MIRK methods have significantly lower computational cost for a given order. As an example, a symmetric, A-stable MIRK method with $s=3$ stages and of order $p = 4$ is given by 
\begin{align*}
  &k_1 = f(y_n), \quad
  k_2 = f(y_{n+1}), \\
  &k_3 = f\bigg(\frac{1}{2}(y_n + y_{n+1}) + \frac{h}{8}(k_1 - k_2) \bigg),\\
  &y_{n+1} = y_n + \frac{h}{6}(k_1 + k_2 + 4k_3).
\end{align*}

\section{Backward error analysis}
Let $\varphi_{h,f}:\mathbb{R}^n\rightarrow \mathbb{R}^n$ be the $h$-flow of an ODE such that $\varphi_{h,f}(y(t_0)) := y(t_{0} + h)$ for an initial value $y(t_0)$. With this notation, the vector field $f_h(y)$ solving the optimization problem \eqref{inverse_problem} exactly must satisfy
\begin{equation}
    \varphi_{h,f}(y(t_n)) = \Phi_{h,f_h}(y(t_n)),\quad n=0,\dots,N-1.
    \label{imde_def}
\end{equation}
For a given numerical one-step method $\Phi$, the \textit{inverse modified vector field} \cite{zhu2022numerical} $f_h$ could be computed by Taylor expansions. However, since their convergence is not guaranteed, truncated approximations are usually considered. This idea builds on backward error analysis \cite[Ch.\ IX]{hairer2006geometric}, which is used in the case of forward problems ($f$ is known and $y(t)$ is approximated) and instead computes the modified vector field $\tilde f_h$ satisfying $\varphi_{h,\tilde f_h}(y(t_n)) = \Phi_{h,f}(y(t_n))$.

An important result, Theorem 3.2 in \cite{zhu2022numerical}, which is very similar to Theorem 1.2 in \cite[Ch.\ IX]{hairer2006geometric}, states that if the method $\Phi_{h,f}$ is of order $p$, then the inverse modified vector field is a truncation of the true vector field, given by
\begin{equation}
  f_h(y) = f(y) + h^pf_p(y) + \dots = f(y) + \mathcal O(h^p).
  \label{imde_order}
\end{equation}

Furthermore, by the triangle inequality, we can express the objective function of the optimization problem \eqref{inverse_problem} in a given point $y(t_{n})$ by
\begin{equation*}
\begin{split}
 \big \| y(t_{n+1}) - \Phi_{h,f_{\theta}}(y(t_n))   \big \| \leq & \,  \big \| \varphi_{h,f}(y(t_n)) - \Phi_{h,f_h}(y(t_n))  \big \| \\& \,  + \big \| \Phi_{h,f_h}(y(t_n))  - \Phi_{h,f_{\theta}}(y(t_n))  \big \|
 \end{split}
\end{equation*}
In the case of formal analysis where we do not consider convergence issues and truncated approximations, the first term is zero by the definition of $f_h$ in \eqref{imde_def}. Thus it is evident that the approximated vector field will approach the inverse modified vector field as the optimization objective tends to zero. Then, by Equation \eqref{imde_order} it is clear that $f_{\theta}(y)$ will learn an approximation of $f(y)$ up to a truncation $\mathcal O(h^p)$, which motivates using an integrator of high order.

 \section{Numerical experiments}
In this section, MIRK methods of order $2 \leq p\leq 6$, denoted by MIRK$p$ in the plots, in addition to the classic fourth-order Runge--Kutta method (RK$4$), is utilized for the temporal discretization in the training of Hamiltonian neural networks. We train on samples $y(t_n)$, for $t_n \in [0,20]$, from solutions of the double pendulum (DP) problem with the Hamiltonian
    \begin{equation*}
      H(y_1,y_2,y_3,y_4) = \frac{\frac{1}{2}y_3^2 + y_4^2 - y_3y_4\cos(y_1-y_2)}{1+\sin^2(y_1-y_2)} - 2\cos(y_1) -\cos(y_2).
    \end{equation*}
In addition, we consider the highly oscillatory Fermi--Pasta--Ulam--Tsingou (FPUT) problem with $m=1$, meaning $y(t)\in\R^4$, and $\omega = 2$ as formulated in \cite[Ch.\ I.5]{hairer2006geometric}. For both Hamiltonian systems, the data $S_{N}=\{ y(t_{i}) \}_{i=0}^{N}$ is found by integrating the system using \textit{DOP853} \cite{DORMAND198019} with a tolerance of $10^{-15}$ for the following step sizes and number of steps: $(h,N) = (2,10),(1,20),(0.5,40)$. The initial values used are $y^{\text{DP}}_0 =[-0.1,0.5,-0.3,0.1]^T$ and $y^{\text{FPUT}}_0 = [0.2,0.4,-0.3,0.5]^T$ . The results for $[y(t)]_3$ are illustrated in Figure \ref{rollout_time}. 

  \begin{figure}
    \centering
 \hspace{50pt}   \text{Double pendulum} \hspace{70pt} \text{Fermi--Pasta--Ulam--Tsingou} \vspace{-5pt}
    \includegraphics[trim=20 34 70 0,clip,width=0.485\textwidth]{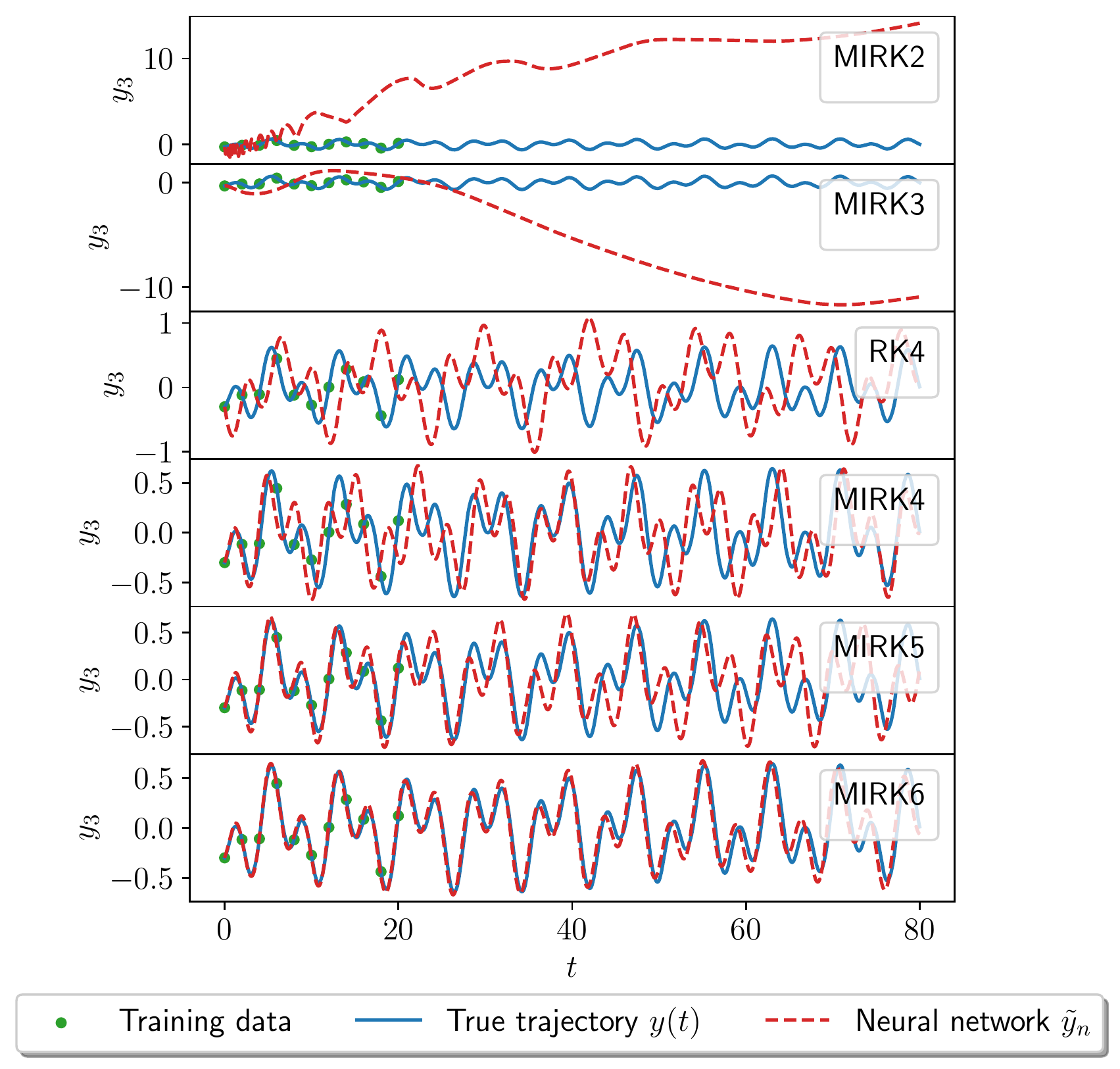}
    \includegraphics[trim=20 34 70 0,clip,width=0.485\textwidth]{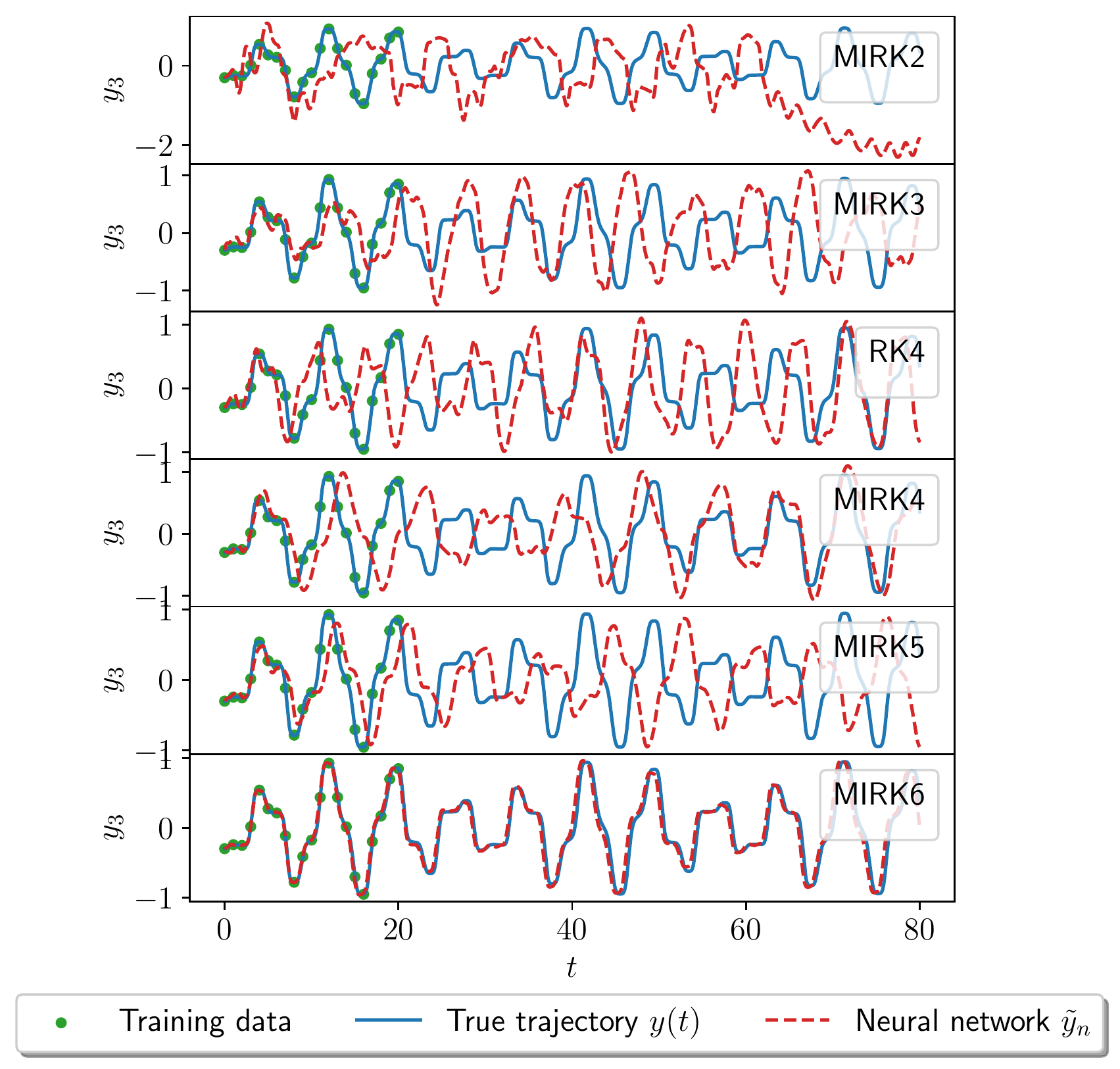}\vspace{2pt}\hspace{13pt}
\includegraphics[trim=5 0 0 430,clip,width=0.6\textwidth]{result_time_multiple_fput.pdf}
\caption{Result when integrating over the learned vector fields when training on data from the double pendulum (left, $N = 10$) and the Fermi--Pasta--Ulam--Tsingou (right, $N = 20$) Hamiltonian.}
\label{rollout_time}
    \end{figure}

 After using the specified integrators in training, approximated solutions $\tilde y_n$ are computed for each learned vector field $f_{\theta}$ again using DOP853, but now with step size and number of steps given by $(h_{\text{test}},N_{\text{test}}) = (\frac{h}{20},4\cdot 20N)$, enabling the computation of the interpolation and extrapolation error:
  \begin{equation}
  \label{flow_error}
    e^{l}(\tilde y) = \frac{1}{M+1} \sum_{n=0}^M  \|\tilde y_n - y(t_n) \|_2, \quad t_n \in Q^l, \quad M = |Q^l|-1.
  \end{equation}
Here $\tilde y_{n+1} := \Phi_{h,f_{\theta}}(\tilde y_n)$,  and $l \in \{i,e\}$ denotes interpolation or extrapolation:  $Q^i = \{hn : \; 0 \leq hn \leq 20, n \in \mathbb Z_+  \}$ and\ $Q^e = \{hn : \; 20 \leq hn \leq 80, n \in \mathbb Z_+  \}$, with $h = h_{\text{test}}$. In addition, the error of the learned Hamiltonian is computed along the true trajectory $y(t_n)$ by
  \begin{equation}
  \label{flow_error}
  \begin{split}
    \overline{ e(H_{\theta}) } &= \frac{1}{M+1} \sum_{n=0}^M  H(y(t_n)) - H_{\theta}(y(t_n)),  \\
e(H_{\theta}) &= \frac{1}{M+1} \sum_{n=0}^M  \bigg | H(y(t_n)) - H_{\theta}(y(t_n)) - \overline{ e(H_{\theta}) } \bigg|,
    \end{split}
  \end{equation}
for $t_n \in Q_i \cup Q_e$. The mean is subtracted since the Hamiltonian is only trained by its gradient $\nabla H_{\theta}$. The error terms are shown in Figure \ref{error}.

  \begin{figure}
    \centering
 \hspace{35pt}   \text{Double pendulum} \hspace{60pt} \text{Fermi--Pasta--Ulam--Tsingou}  \vspace{-5pt}
    \includegraphics[trim=7 55 0 0,clip,width=0.45\textwidth]{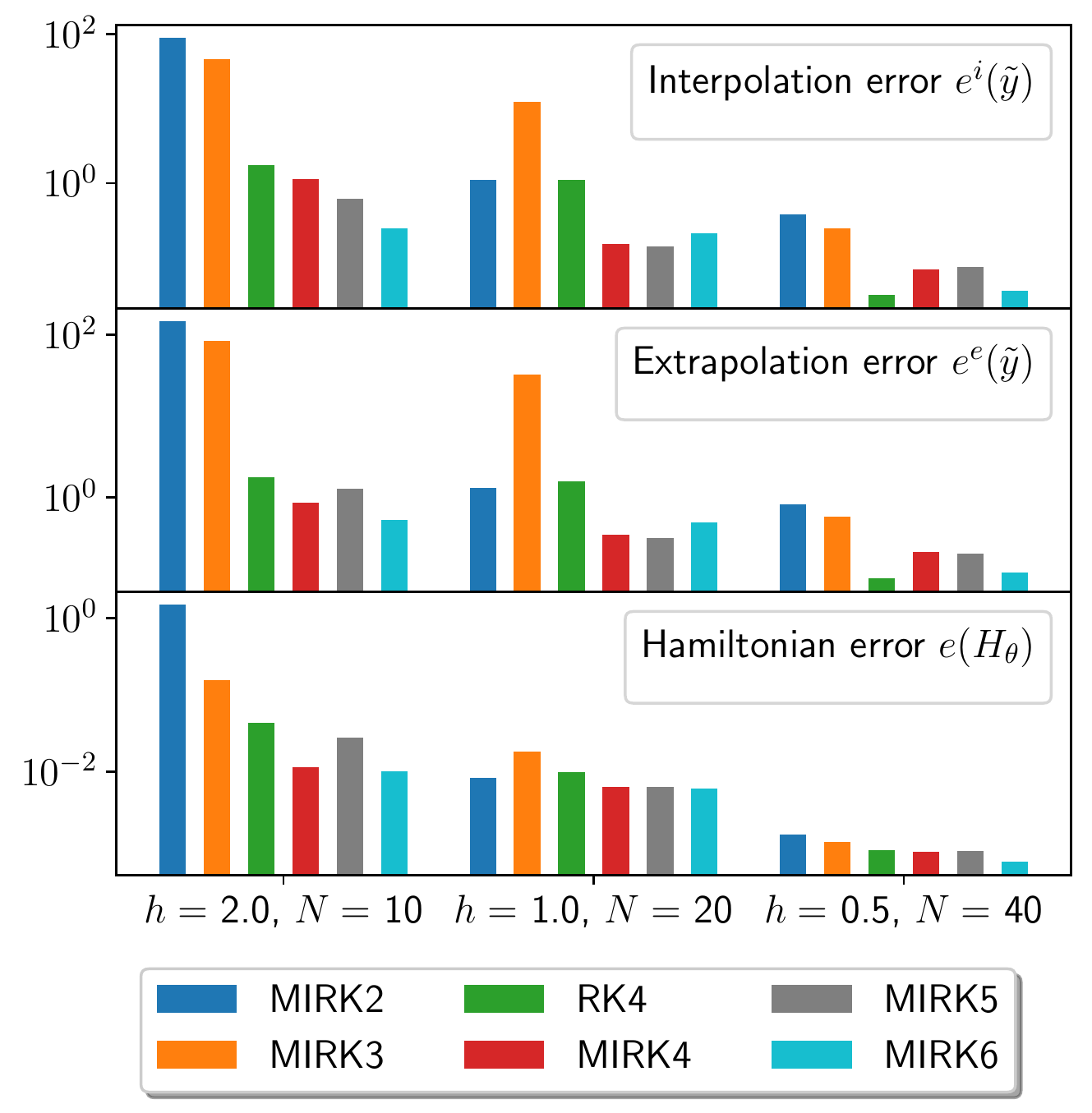}
    \includegraphics[trim=7 55 0 0,clip,width=0.45\textwidth]{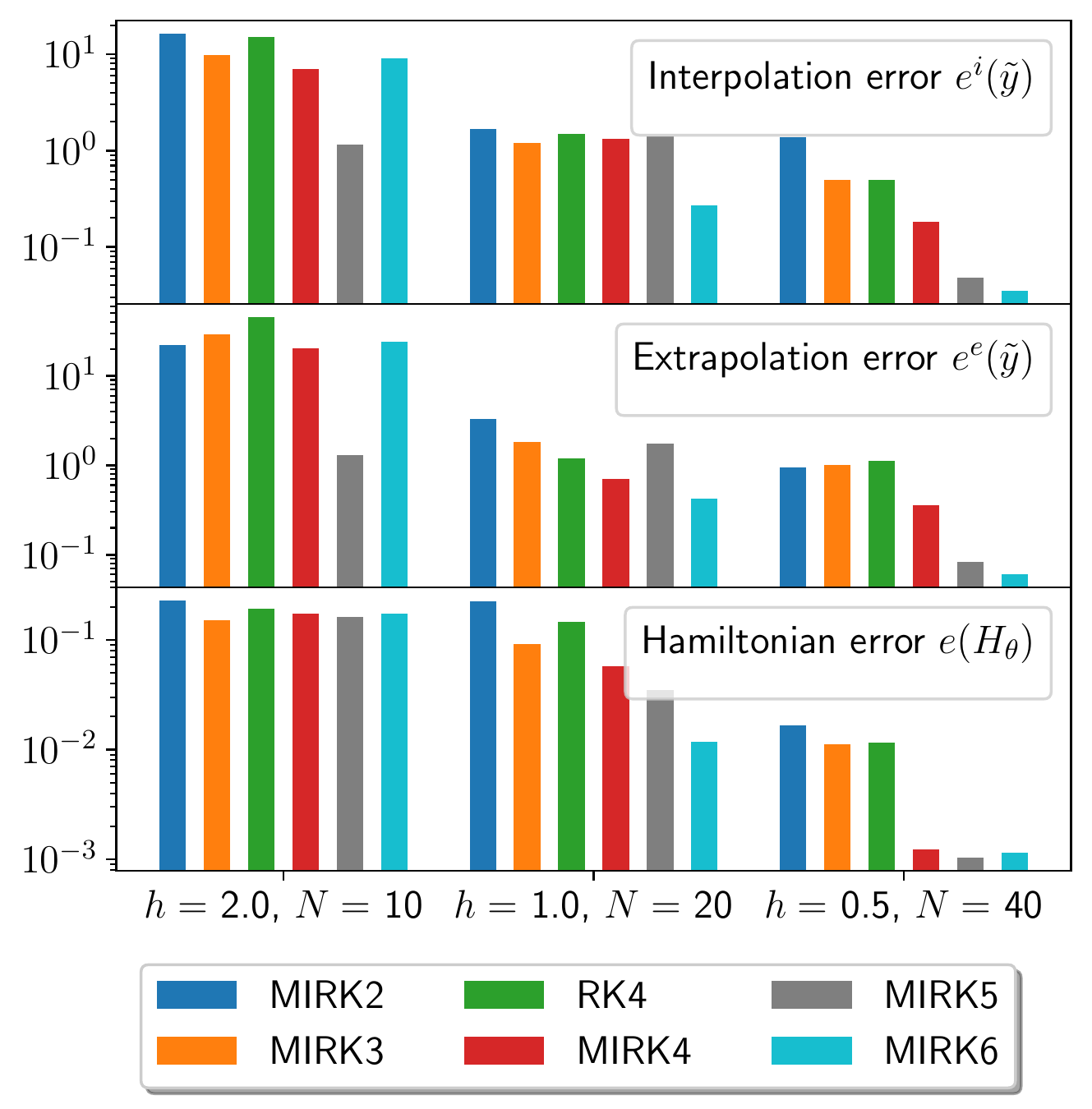}\vspace{6pt}
    \includegraphics[trim=40 0 25 335,clip,width=0.40\textwidth]{error_fput.pdf}
    \caption{Errors in interpolation, extrapolation and the Hamiltonian for the double pendulum (left) and the Fermi--Pasta--Ulam--Tsingou problem (right).}
    \label{error}
    \end{figure}

For both test problems the Hamiltonian neural networks have $3$ layers with a width of $100$ neurons and $\text{tanh}( \cdot )$ as the activation function. Experiments are implemented using PyTorch \cite{pytorch} and the optimization problem is solved using the quasi-Newton L-BFGS algorithm \cite{nocedal1999numerical} for $100$ epochs without batching. The implementation of the experiments could be found in the following repository \href{https://github.com/hakonnoren/learning_hamiltonian_mirk}{github.com/hakonnoren/learning{\_}hamiltonian{\_}mirk}.

\section{Conclusion}
The mono-implicit Runge--Kutta methods enable the combination of high order and computationally efficient training of Hamiltonian neural networks. The importance of high order is demonstrated by the remarkable capacity of MIRK$6$ in learning a trajectory of the chaotic double pendulum and the Fermi--Pasta--Ulam--Tsingou Hamiltonian systems from just $11$ and $21$ points, see Figure \ref{rollout_time}. In most cases the error, displayed in Figure \ref{error}, is decreasing when increasing the order. Additionally MIRK$4$ displays superior performance comparing with the explicit method RK$4$ of same order. Even though the numerical experiments show promising results, the theoretical error analysis in this work is rudimentary at best. Future work should consider this is greater detail, perhaps along the lines of \cite{zhu2022numerical}.

\section*{Acknowledgments}
The author wishes to express gratitude to Elena Celledoni and Sølve Eidnes for constructive discussions and helpful suggestions while
working on this paper.

%
%
%
 \bibliographystyle{splncs04}
 \bibliography{ref}

\end{document}